\documentclass[11 pt]{amsart}

\usepackage[utf8]{inputenc}
\usepackage[english]{babel}
\usepackage{morefloats}
\usepackage{amssymb} 
\usepackage{amsfonts} 
\usepackage{amsmath}
\usepackage{mathrsfs}
\usepackage{amsthm} 
\usepackage{epsfig} 
\usepackage{color}
\usepackage{amscd}
\usepackage[all]{xy}
\usepackage{subfigure}

\newtheorem{lemma}{Lemma}[section]
\newtheorem{teo}[lemma]{Theorem}
\newtheorem{prop}[lemma]{Proposition}
\newtheorem{cor}[lemma]{Corollary}

\theoremstyle{definition}

\newtheorem{quest}[lemma]{Question}

\newtheorem{example}[lemma]{Example}

\theoremstyle{remark}

\newcommand{\Iso}{{\rm Isom}}

\newcommand{\matR} {\ensuremath {\mathbb{R}}}

\newcommand{\matF} {\ensuremath {\mathbb{F}}}

\newcommand{\matH} {\ensuremath {\mathbb{H}}}

\newcommand{\matZ} {\ensuremath {\mathbb{Z}}}

\newcommand{\scrM}{\ensuremath {\mathscr M}}
\newcommand{\scrN}{\ensuremath {\mathscr N}}
\newcommand{\scrD}{\ensuremath {\mathscr D}}
\newcommand{\scrZ}{\ensuremath {\mathscr Z}}
\newcommand{\scrP}{\ensuremath {\mathscr P}}
\newcommand{\scrQ}{\ensuremath {\mathscr Q}}

\newcommand{\isom}{\cong}

\newcommand{\Vol}{{\rm Vol}}

\newcommand{\prfend}{{\hfill\hbox{$\square$}\vspace{2pt}}}

\author{Alexander Kolpakov}
\address{University of Toronto, Department of Mathematics, 40 St. George Street, Toronto ON, M5S 2E4, Canada}
\email{kolpakov dot alexander at gmail dot com}

\author{Bruno Martelli}
\address{Universit\`{a} di Pisa, Dipartimento di Matematica ``Tonelli'', Largo Pon\-te\-cor\-vo 5, 56127 Pisa, Italy}
\email{martelli at dm dot unipi dot it}

\author{Steven Tschantz}
\address{Vanderbilt University, Department of Mathematics, 1326 Stevenson Center, Nashville, TN 37240, USA}
\email{steven dot tschantz at vanderbilt dot edu}

\thanks{The first named author is supported by the SNSF researcher scholarship P300P2-151316. The second named author was supported by the Italian FIRB project ``Geometry and topology of low-dimensional manifolds'', RBFR10GHHH}

\title[]{Some hyperbolic three-manifolds \\ that bound geometrically}

\begin{document}

\begin{abstract}
A closed connected hyperbolic $n$-manifold \emph{bounds geometrically} if it is isometric to the geodesic boundary of a compact hyperbolic $(n+1)$-manifold. A.~Reid and D.~Long have shown by arithmetic methods the existence of infinitely many manifolds that bound geometrically in every dimension.

We construct here infinitely many explicit examples in dimension $n=3$ using right-angled dodecahedra and 120-cells and a simple colouring technique introduced by M.~Davis and T.~Januszkiewicz. Namely, for every $k\geqslant 1$,  we build an orientable compact closed 3-manifold tessellated by $16k$ right-angled dodecahedra that bounds a 4-manifold tessellated by $32k$ right-angled 120-cells.

A notable feature of this family is that the ratio between the volumes of the 4-manifolds and their boundary components is constant and, in particular, bounded. 
\end{abstract}

\maketitle

\section{Introduction}\label{section:1}

The study of hyperbolic manifolds that bound geometrically dates back to the works of D.~Long, A.~Reid \cite{LR1, LR2} and B. Niemershiem \cite{N}, motivated by a preceding work of M.~Gromov \cite{G1, G2} and a question by F. Farrell and S. Zdravkovska \cite{FZ}. This question is also related to hyperbolic instantons, as described by J.~Ratcliffe and S.~Tschantz \cite{RT1, RT2}. In particular, the following problems are of particular interest:
\begin{quest}\label{question1}
Which compact orientable hyperbolic $n$-manifold $\mathscr{N}$ can represent the totally geodesic boundary of a compact orientable hyperbolic $(n+1)$-manifold $\mathscr{M}$?
\end{quest}
\begin{quest}\label{question2}
Which compact orientable flat $n$-manifold $\mathscr{N}$ can represent the cusp section of a single-cusped orientable hyperbolic $(n+1)$-manifold $\mathscr{M}$?
\end{quest}
Once there exist such manifolds $\mathscr{N}$ and $\mathscr{M}$, we say that $\mathscr{N}$ \emph{bounds $\mathscr{M}$ geometrically}. In this note, we shall concentrate on Question~\ref{question1}, devoted to compact geometric boundaries. The recent progress on Question~\ref{question2}, that involves cusp sections,  is indicated by \cite{KM, LR2, McR, McRRS}. However, this is still an open problem in dimensions $\geq 5$. On the other hand, by a result of M. Stover \cite{Stover}, an arithmetic orbifold in dimension $\geq 30$ cannot have a single cusp.

In \cite{LR1}, D. Long and A. Reid have shown that many closed hyperbolic 3-manifolds do not bound geometrically: a necessary condition is that the eta invariant of the 3-manifold must be an integer. The first known closed hyperbolic 3-manifold that bounds geometrically was constructed by J. Ratcliffe and S. Tschantz in \cite{RT1} and has volume of order $200$.

Then, D. Long and A. Reid produced in \cite{LR3}, by arithmetic techniques, infinitely many orientable hyperbolic $n$-manifolds $\scrN$ that bound geometrically an $(n+1)$-manifold $\scrM$, in every dimension $n\geqslant 2$. Every such manifold $\scrN$ is obtained as a cover of some $n$-orbifold $O_\scrN$ geodesically immersed in a suitable $(n+1)$-orbifold $O_\scrM$. 

In this paper, we construct  an explicit infinite family in dimension $n=3$, via a similar covering technique where the roles of $O_\mathscr N$ and $O_\mathscr M$ are played by the right-angled dodecahedron $\mathscr D$ and 120-cell $\mathscr Z$. These two compact Coxeter right-angled regular polytopes exist in $\matH^3$ and $\matH^4$ respectively, and the first is a facet of the second. The existence of suitable finite covers is guaranteed here by assigning appropriate \emph{colourings} to their facets, following A.~Vesnin \cite{Vesnin87, Vesnin}, M.~Davis and T.~Januszkiewicz \cite{DJ}, I.~Izmestiev \cite{I}. 

A colouring determines a manifold covering, and the main factual observation is that a colouring of the dodecahedron $\mathscr{D}$ can be enhanced in a suitable way to a colouring of the right-angled hyperbolic Coxeter $120$-cell $\mathscr{Z}$. 
We produce in this way a degree-$32$ orientable cover of $\mathscr Z$ that contains four copies of a non-orientable degree-$8$ cover of $\mathscr{D}$. By cutting along one such non-orientable geodesic submanifold we get a hyperbolic four-manifold $\scrN_1$ with connected geodesic boundary $\scrM_1 = \partial \scrN_1$.

The colouring technique applied to a single polytope can produce only finitely many manifolds. To get infinitely many examples we assemble $n$ copies of $\scrD$ and $\scrZ$ to get more complicated right-angled polytopes, to which the above construction easily extends. We finally obtain the following. Let $V_\scrD \approx 4.3062...$ and $V_\scrZ = \frac{34}3 \pi^2$ be the volumes of $\scrD$ and $\scrZ$, respectively.

\vspace*{-0.01in}

\begin{teo}\label{theorem}
For every $n\geqslant 1$ there exists an orientable compact hyperbolic three-manifold $\scrN_n$ of volume $16 n V_\scrD$ which bounds geometrically an orientable compact hyperbolic four-manifold $\scrM_n$ of volume $32 n V_\scrZ$. 

The manifolds $\scrN_n$ and $\scrM_n$ are tessellated respectively by $16 n$ right-angled dodecahedra and $32 n$ right-angled $120$-cells. 
\end{teo}

An interesting feature of this construction is that it provides manifolds $\scrN_n$ and $\scrM_n$ of controlled volume. In particular, we deduce the following.

\begin{cor}
There are infinitely many hyperbolic three-manifolds $\scrN$ that bound geometrically some hyperbolic $\scrM$ with constant ratio: 
$$\frac{\Vol(\scrM)}{\Vol (\scrN)} =  \frac{2V_\scrZ}{V_\scrD} < 53.$$
\end{cor}

The manifold $\scrM_1$ has volume $16V_\scrD \approx 68.8992$ and is to our knowledge the smallest closed hyperbolic 3-manifold known to bound geometrically.

\subsection*{Structure of the paper.} In Section \ref{section:2} we introduce right-angled polytopes as orbifolds, and a simple colouring technique from \cite{DJ} that produces manifold coverings of small degree. Then we show how this colouring technique passes easily from dimension $n$ to $n+1$ and conversely, and may be used to produce $n$-manifolds that bound geometrically when $n=3$. In Section \ref{section:3} we assemble dodecahedra and 120-cells to produce the manifolds $\scrN_k$ and $\scrM_k$ of Theorem \ref{theorem}.

\section{Colorings and covers of Coxeter orbifolds}\label{section:2}
A right-angled hyperbolic polytope $\scrP\subset\matH^n$ may be interpreted as an orbifold with mirror boundary, where the mirrors correspond to its facets. As such an orbifold, it has a plenty of manifold coverings. A few of them may be constructed by colouring appropriately the facets of $\scrP$ as shown in \cite{DJ, GS}. 

\subsection{Colourings and manifold covers.} 
Let $\scrP \subset \mathbb{H}^n$ be a convex compact right-angled polytope. Such objects exist only if $2\leqslant n \leqslant 4$, see \cite{PV}; the two important basic examples we consider here are the right-angled dodecahedron $\scrD \subset \matH^3$ and the right-angled 120-cell $\scrZ \subset\matH^4$. 

We consider $\scrP$ as an orbifold $\matH^n /_\Gamma$. The group $\Gamma$ is a right-angled Coxeter group that may be presented as
$$\Gamma = \langle\ r_F \ |\ r_F^2, [r_F, r_{F'}] \ \rangle$$
where $F$ varies over all the facets of $\scrP$ and the pair $F, F'$ varies over all the pairs of adjacent facets. The isometry $r_F\in\Iso (\matH^n)$ is a reflection in the hyperplane containing $F$.

A right-angled polyhedron $\scrP$ is \emph{simple} \cite[Theorem~1.8]{Vinberg}, which means that it looks combinatorially at every vertex $v$ like the origin of an orthant in $\matR^n$. In particular, $v$ is the intersection of exactly $n$ facets.

Let $V$ be a finite-dimensional vector space over $\matF_2$, thus isomorphic to $\matF_2^s$ for some $s$. A \emph{$V$-colouring} (or simply, a colouring) $\lambda$ is the assignment of a vector $\lambda_F \in V$ to each facet $F$ of $\scrP$ (called its \emph{colour}) such that the following holds: at every vertex $v$, the $n$ colours assigned to the $n$ adjacent facets around $v$ are linearly independent vectors in $V$.

A colouring induces a group homomorphism $\lambda\colon\Gamma \to V$, defined by sending $r_F$ to $\lambda_F$ for every facet $F$. Its kernel $\Gamma_\lambda = \ker \lambda$ 
is a subgroup of $\Gamma$ which determines an orbifold $M_\lambda = \matH^n/_{\Gamma_\lambda}$ covering $\scrP$.

\begin{prop} The orbifold $M_\lambda$ is a manifold.
\end{prop}
\begin{proof}
We follow \cite[Lemma 1]{Vesnin87}. A torsion element in $\Gamma$ fixes some face $F$ of the tessellation of $\matH^n$ obtained by reflecting $\scrP$ in its own facets. Up to conjugacy, we can suppose that $F\subset\scrP$. The stabilizer of $F$ is generated by the reflections in the facets containing $F$ (see \cite[Theorem~12.3.4]{Davis}) and is hence mapped injectively into $V$ by $\lambda$. Thus, $\Gamma_\lambda$ is torsion-free and $M_\lambda$ is a manifold.
\end{proof}

At each vertex $v$ the colours $\lambda_F$ of the $n$ incident facets $F$ are independent: therefore the image of $\lambda$ has dimension at least $n$ and
the covering $M_\lambda \to \scrP$ has degree
$|\Gamma:\Gamma_\lambda| \geqslant 2^n$; if the equality holds the manifold $M_\lambda$ is called a \textit{small cover} of $\scrP$. The manifold coverings of $\scrP$ of smallest degree are precisely its small covers, see \cite[Proposition 2.1]{GS}.

We say that the colouring \emph{spans} $V$ if the vectors $\lambda_F$ span $V$ as $F$ varies, which is equivalent to the map $\lambda \colon \Gamma \to V$ being surjective. In that case the covering $M_\lambda \to \scrP$ has degree $|\Gamma : \Gamma_\lambda| = 2^{\, \dim V}$. 

\subsection{$k$-colourings}
Here, we give an example: recall that a \emph{$k$-colouring} of a polytope is the assignment of a colour from the set $\{1,\ldots, k\}$ to each facet so that two adjacent facets have distinct colours. A $k$-colouring for $\scrP$ produces an $\matF_2^k$-colouring that spans $\matF_2^k$: simply replace each colour $i\in\{1,\ldots, k\}$ with the element $e_i$ of the canonical basis for $\matF_2^k$.

\begin{example} \label{D}
The dodecahedron has precisely one four-colouring, up to symmetries. This induces an $\matF_2^3$-colouring (of lower dimension $3$ rather than $4$) on the hyperbolic right-angled dodecahedron $\scrD$ described in \cite{Vesnin87} and hence a manifold covering having degree $2^3=8$.
\end{example}

\begin{example} \label{Z}
The 120-cell has a five-colouring (in fact, ten five-colourings up to  symmetries \cite{F}). Each produces a manifold covering of the hyperbolic right-angled 120-cell $\mathscr{Z}$ of degree $2^5=32$. 
\end{example}

\subsection{Orientable coverings}
The following lemma gives an orientability criterion analogous to that for small covers in \cite{NN} or L\"obell manifolds in \cite[Lemma 2]{Vesnin87}. Let $\lambda$ be a $V$-colouring of a right-angled polytope $\scrP$.

\begin{lemma}\label{lambda-orientable:lemma}
Suppose $\lambda$ spans $V$. The manifold $M_\lambda$ is orientable if and only if, for some isomorphism $V\isom \matF_2^s$, each colour $\lambda_F$ has an odd number of $1$'s.
\end{lemma}
\begin{proof}
Let $\Gamma^+\triangleleft \Gamma$ be the index two subgroup consisting of orientation-preserving isometries. Then $\Gamma^+$ is the kernel of the homomorphism $\phi\colon\Gamma \to \matF_2$ that sends $r_F$ to $1$ for every facet $F$.

The manifold $M_\lambda$ is orientable if and only if $\Gamma_\lambda$ is contained in $\Gamma^+$, and this in turn holds if and only if there is a homomorphism $\chi: V \rightarrow \matF_2$ such that $\phi = \chi \circ \lambda$. The latter is equivalent to the existence of an isomorphism $V\isom \matF_2^s$ that transforms $\lambda_F$ into a vector with an odd number of $1$'s for each $F$. Indeed, if such an isomorphism exists, then $\chi$ can be taken to be the sum of the coordinates of a vector. 

Conversely, suppose such an isomorphism does not exist. Since the vectors $\lambda_F$ span $V$ we may take some of them as a basis for $V$ and write them as $e_1 = (1,0,0,\dots,0)$, $e_2 = (0,1,0,\dots,0)$, $\dots$, $e_s = (0,\dots,0,0,1)$. By hypothesis, there exists a facet $F$ such that $\lambda_F$ has an even number of $1$'s. Up to reordering, we may write $\lambda_F = \sum^{2k}_{i=1} e_i$ for some $k$. Now we can see that the homomorphism $\chi$ does not exist, since its existence would imply $1 = \phi(r_F) = \sum^{2k}_{i=1} \phi(e_i) = \sum^{2k}_{i=1} 1 = 0$. 
\end{proof}

\begin{cor}\label{lambda-orientable:corollary}
Let there be facets $F$, $F^\prime$ and $F^{\prime\prime}$ of $\scrP$ such that $\lambda_{F} + \lambda_{F^{\prime}} + \lambda_{F^{\prime\prime}} = \mathbf{0}$. Then $\lambda$ is a non-orientable colouring.
\end{cor}
\begin{proof}
For a vector $v = (v_1, v_2, \dots, v_s) \in \matF_2^s$ let $\epsilon(v) := \sum^s_{i=1} v_i$. Suppose that $\lambda$ is orientable, so there exists an isomorphism $V \cong \matF^s_2$, such that each $\lambda_{F}$ has an odd number of $1$'s. Then we arrive at a contradiction, since $0 = \epsilon(\mathbf{0}) = \epsilon(\lambda_{F} + \lambda_{F^\prime} + \lambda_{F^{\prime\prime}}) = \epsilon(\lambda_F) + \epsilon(\lambda_{F^\prime}) + \epsilon(\lambda_{F^{\prime\prime}}) = 1$. 
\end{proof}

\begin{example} The manifolds in Examples \ref{D} and \ref{Z} are orientable.
\end{example}

\begin{example} Consider the $25$ small covers of the hyperbolic right-angled dodecahedron $\mathscr{D} $ found by A.~Garrison and R.~Scott in \cite{GS}. The list is complete, up to isometries between the corresponding manifolds. Using the orientability criterion one sees immediately from \cite[Table~1]{GS} that $24$ of them are non-orientable and exactly $1$ is orientable and corresponds to Example \ref{D}.
Another example carried out in \cite{GS} is a small cover of the hyperbolic right-angled $120$-cell $\mathscr{Z}$. This cover is again non-orientable. There is no classification of small covers of $\mathscr{Z}$ known at present.
\end{example}

\subsection{Induced colouring}
A facet $F$ of a $n$-dimensional right-angled polytope $\scrP\subset\matH^n$ is a $(n-1)$-dimensional right-angled polytope. A $V$-colouring $\lambda$ of $\scrP$ induces a $W$-colouring $\mu$ of $F$ with $W = V/_{\langle \lambda_F \rangle}$: simply assign to every face of $F$ the colour of the facet of $\scrP$ that is incident to it. The following lemma generalises \cite[Proposition~2.3]{GS}. 

\begin{lemma}\label{colouring-geodesic:lemma}
The manifold $M_\mu$ is contained in $M_\lambda$ as a totally geodesic sub-manifold, so that the cover $M_\lambda \to \scrP$ restricts to the cover $M_\mu \to F$.
\end{lemma}
\begin{proof}
Let $\Gamma$ be the Coxeter group of $\scrP$ and $\matH^{n-1}\subset \matH^n$ be the hyperplane containing $F$. We regard $F$ as the orbifold $\matH^{n-1}/_{\Gamma_F}$ where $\Gamma_F$ is the Coxeter group of $F$. The following natural diagram commutes:
$$
\xymatrix{ 
0 \ar@{->}[r] & \langle r_F \rangle \ar@{->}[r] & \Gamma\cap {\rm Stab}(\matH^{n-1}) \ar@{->}^{\qquad f}[r] \ar@{->}_\lambda[d] & \Gamma_F \ar@{->}[r] \ar@{->}^\mu[d] & 0 \\
 & & V \ar@{->}_\pi[r] & V /_{\langle \lambda_F \rangle}&
 }
$$
The first line is an exact sequence. We deduce easily that $f$ restricts to an isomorphism 
$$f\colon \ker \lambda \cap {\rm Stab}(\matH^{n-1}) \longrightarrow \ker \mu.$$
Hence $M_\mu = \matH^{n-1}/_{\Gamma_\lambda \cap {\rm Stab}(\matH^{n-1})}$ is naturally contained in $M_\lambda = \matH^n/_{\Gamma_\lambda}$.
\end{proof}

The pre-image of $F$ in $M_\lambda$ with respect to the regular covering $M_\lambda \to \scrP$ consists of possibly several copies of $M_\mu$.

\subsection{Extended colouring}
Conversely, we can also extend a colouring from a facet to the whole polytope. 
We say that two colourings $\lambda$ and $\lambda'$ on $\scrP$ are \emph{equivalent} if they have isomorphic kernels $\Gamma_\lambda \cong \Gamma_{\lambda'}$ (cf. the definition before \cite[Proposition~2.4]{GS}). 

\begin{prop} \label{extend:prop}
Let $F$ be a facet of a compact right-angled polytope $\scrP \subset \mathbb{H}^n$. Every colouring of $F$ is equivalent to one induced by an orientable colouring of $\scrP$.
\end{prop}
\begin{proof}
Let $\lambda$ be a $V$-colouring of the facet $F$. Fix an isomorphism $V \isom \matF_2^s$. Define $W= \matF_2 \oplus \matF_2^s \oplus \matF_2^f$ where $f$ is the number of facets of $\scrP$ that are not adjacent to $F$. We define a $W$-colouring $\mu$ of $\scrP$ as follows:
\begin{itemize}
\item set $\mu_F = (1, \mathbf{0}, \mathbf{0})$;
\item set $\mu_G = (\epsilon (\lambda_{G\cap F})+1, \lambda_{G\cap F}, \mathbf{0})$ where $\epsilon(v) = \sum^{s}_{i=1} v_i$, for every facet $G$ adjacent to $F$;
\item set $\mu_{G_i} = (\mathbf{0}, \mathbf{0}, e_i)$ for the remaining facets $G_1,\ldots G_f$. 
\end{itemize}

Indeed, the map $\mu$ is a colouring: the linear independence condition is satisfied at each vertex. Moreover, each vector $\mu_F, \mu_G, \mu_{G_i}$ has an odd number of $1$'s.  Finally, by construction $\mu|_F$ is equivalent to $\lambda$. 
\end{proof}

We call the colouring $\mu$ an \textit{extension} of $\lambda$. 

\subsection{A more efficient extension}
Proposition \ref{extend:prop} shows how to extend a $V$-colouring of a facet $F$ to an orientable $W$-colouring of the polytope $\scrP$. The proof shows that the dimension of $W$ can grow considerably during the process, since $\dim W = \dim V + 1 + f$ where $f$ is the number of facets of $\scrP$ not adjacent to $F$.

We may use Proposition \ref{extend:prop} with $\scrP$ being the right-angled 120-cell and $F$ its dodecahedral facet. However, in this case we could find examples where both $V$ and $W$ have smaller dimension via computer. The following was proved by using ``Mathematica''.

\begin{prop} \label{5:prop}
Each of the $24$ non-orientable $\matF^3_2$-colourings of $\scrD$ from \cite[Table~1]{GS} is equivalent to one induced by an orientable $\matF^5_2$-colouring of $\scrZ$.
\end{prop}
\begin{proof}
The Mathematica program code given in \cite{Tschantz} takes a non-orientable $\matF_2^3$-colouring of $\scrD$ and produces an orientable $\matF_2^5$-colouring of $\scrZ$, as required.

Each vector $v = (v_1, v_2, \dots, v_s) \in \matF^s_2$ is encoded by the binary number $n_v = v_1\cdot 2^0 + v_2\cdot 2^1 + \dots + v_s\cdot 2^{s-1}$. Let $P_0 := \scrD$, be the right-angled dodecahedron. We enumerate its faces exactly as shown in \cite[Figure~3]{GS} and the corresponding 12-tuple of numbers encodes its colouring. Let $P_i$, $i=1,\dots,12$ be the dodecahedral facets incident to $P_0$ at the respective faces $F_i$, $i=1,\dots,12$. We start extending the colouring of $P_0 := \scrD$ as follows:
\begin{itemize}
\item set $\lambda_{P_0} = (0, 0, 0, 0, 1)$;
\item if $\mu_{F_i} = v = (v_1, v_2, v_3)$, then $\lambda_{P_i} := (v_1, v_2, v_3, 0, \epsilon(v)+1)$.
\end{itemize} 

We obtain a $13$-tuple, which is the initial segment of the colouring of $\scrZ$. Then the Mathematica code \cite{Tschantz} attempts to produce an orientable $120$-tuple, which encodes the entire colouring.
\end{proof}

We were not able to find any orientable $\matF^4_2$-colouring of $\scrZ$ extending a non-orientable $\matF^3_2$-colouring of $\scrD$: our examples are not small covers. 

Let now $M_\lambda$ be the manifold obtained by one $\matF^5_2$-colouring $\lambda$ of $\scrZ$: it covers $\scrZ$ with degree $2^5 = 32$. The colouring $\lambda$ restricts to a non-orientable colouring $\mu$ of the facet $\scrD$, which gives rise, by Lemma \ref{colouring-geodesic:lemma}, to a non-orientable codimension-1 geodesic submanifold $M_\mu \subset M_\lambda$, covering $\scrD$ with degree $2^3 = 8$. 

In complete analogy to \cite[Lemma~3.2]{LR3}, by cutting $M_\lambda$ along $M_\mu$ we get a compact orientable hyperbolic manifold tessellated by $32$ right-angled 120-cells having a geodesic boundary isometric to a connected manifold $\widetilde {M_\mu}$ that double-covers $M_\mu$ and is hence tessellated by $2\cdot 2^3 = 16$ right-angled dodecahedra.

To prove Theorem \ref{theorem} it only remains to extend this argument from one 120-cell to an appropriate assembling of $n \geq 2$ distinct 120-cells.

\section{Assembling right-angled dodecahedra and 120-cells} \label{section:3}
Here, we assemble right-angled dodecahedra and 120-cells in order to construct more complicated convex compact right-angled convex compact polytopes.

\subsection{Connected sum of polytopes.} Let $\scrP_1$ and $\scrP_2$ be two right-angled polytopes in $\mathbb{H}^n$. If there is an isometry between two facets of $\scrP_1$ and $\scrP_2$, we may use it to glue them: the result is a new right-angled polytope in $\matH^n$ which we call a \textit{connected sum} of $\scrP_1$ and $\scrP_2$ along these facets.  

\subsection{Assembling}
An \emph{assembling} of right-angled dodecahedra (or 120-cells) is a right-angled polytope constructed from a finite sequence
$$\scrP_1 \# \scrP_2 \# \scrP_3 \# \ldots \# \scrP_k$$
of connected sums performed from the left to the right, where each $\scrP_i$ is a right-angled dodecahedron (or 120-cell).

\begin{lemma} \label{facet:lemma}
An assembling of $k$ right-angled dodecahedra is a facet of an assembling of $k$ right-angled 120-cells.
\end{lemma}
\begin{proof}
Consider $\matH^3$ inside $\matH^4$ as a geodesic hyperplane. Consider $\scrD\subset \matH^3$ as a facet of $\scrZ \subset \matH^4$. Every time we attach a new copy $\scrD_i$ of $\scrD$ in $\matH^3$, we correspondingly attach a new copy $\scrZ_i$ of $\scrZ$ having $\scrD_i$ as a facet.
\end{proof}

\subsection{Proof of Theorem \ref{theorem}.} We have described all the ingredients necessary to prove Theorem \ref{theorem}.

By Proposition \ref{5:prop}, pick an orientable $\matF^5_2$-colouring of the right-angled 120-cell $\scrZ$ that induces a non-orientable $\matF^3_2$-colouring of one dodecahedral facet $\scrD$.

Then assemble $n$ copies of $\scrD$ as
$$\scrP = \scrD_1 \# \scrD_2 \# \ldots \# \scrD_n$$
and consider, as in Lemma \ref{facet:lemma}, the resulting right-angled polyhedron $\scrP$ as a facet of a a right-angled polytope $\scrQ$ made of $n$ copies of the right-angled 120-cell, each having a $\scrD_i$ as a facet.

Every time we assemble a new copy of $\scrD_i$, we give $\scrD_i$ the colouring of the adjacent dodecahedron, mirrored along the glued pentagonal face, and we do the same with each corresponding new 120-cell $\scrZ_i$. The resulting polytope $\scrQ$ inherits an orientable $\matF_2^5$-colouring $\lambda$ that induces a non-orientable $\matF_2^3$-colouring $\mu$ of $\scrP$, if $\scrP$ is assembled appropriately. Indeed, each of the 24 non-orientable colourings of $\scrD$ has three faces $F$, $F^{\prime}$ and $F^{\prime\prime}$ satisfying the conditions of Corollary~\ref{lambda-orientable:corollary}. Moreover, we can find a fourth face $F^\ast$ which is disjoint from each of them. These properties are easily verifiable by using \cite[Table~1]{GS}. Then, we start assembling $\scrD_i$'s by forming a connected sum along $F^\ast$. Then the colouring of the resulting polytope $\scrP$ again satisfies Corollary~\ref{lambda-orientable:corollary} and hence is non-orientable, as required. 

By cutting $M_\lambda$ along $M_\mu$ we get a compact orientable hyperbolic manifold tessellated by $32n$ right-angled 120-cells having a geodesic boundary that is isometric to a connected manifold $\widetilde {M_\mu}$ that double-covers $M_\mu$ and is hence tessellated by $2\cdot 2^3n = 16n$ right-angled dodecahedra. \prfend

We conclude the paper by providing an example of the construction carried in the proof above.

\begin{figure}
 \begin{center}
  \includegraphics[width = 6 cm]{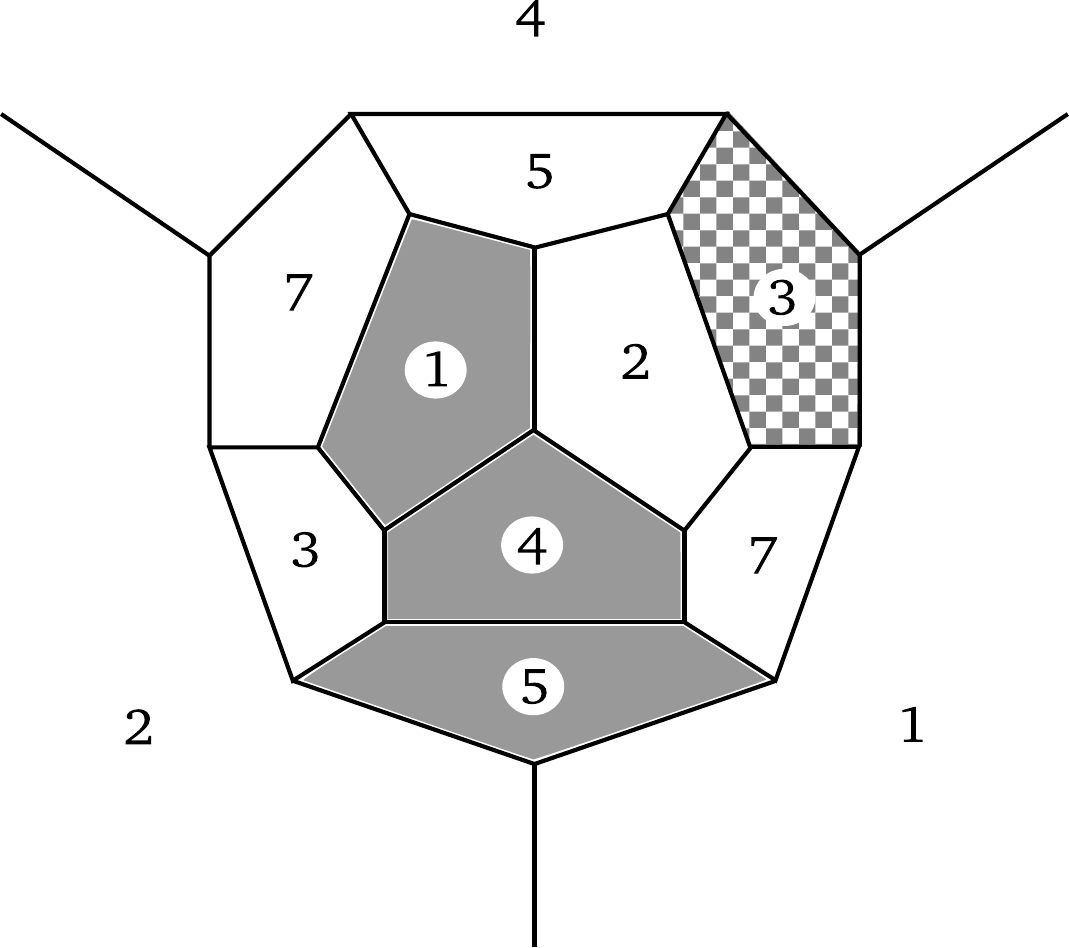}
 \end{center}
 \caption{A non-orientable colouring of $\scrD$ from \cite[Table~1]{GS}. Face colours are encoded by binary numbers}
 \label{dodecahedron-colour:fig}
\end{figure}

\begin{example} 
Let us choose a non-orientable colouring $\mu$ of the dodecahedron $\scrD$ from \cite[Table~1]{GS}, say the one having the maximal symmetry group $(\matZ_2\times \matZ_2)\ltimes \matZ_6$. This is an $\matF^3_2$-colouring depicted in Fig.~\ref{dodecahedron-colour:fig}, with face colours encoded by binary numbers in the decimal range $1,\dots,7$. Here, the grey shaded faces of $\scrD$ are exactly the faces $F$, $F^{\prime}$ and $F^{\prime\prime}$, satisfying the conditions of Corollary~\ref{lambda-orientable:corollary}. The face $F^\ast$, along which we take a connected sum in the proof of Theorem~\ref{theorem}, has a chequerboard shading. Now, we take a connected sum of $\scrD_1 := \scrD$ along $F^\ast$ with its isometric copy $\scrD_2$, having the same colouring. Then we can choose a face of $\scrD_1 \# \scrD_2$, distinct from any of $F$, $F^{\prime}$ or $F^{\prime\prime}$ and continue assembling until we use all $n$ given copies of $\scrD$, which produce a polyhedron $\scrP = \scrD_1\# \dots \# \scrD_n$. The faces $F$, $F^{\prime}$ and $F^{\prime\prime}$ of $\scrD_1$ still remain among those of $\scrP$. Thus, the colouring of $\scrP$ is again non-orientable by Corollary~\ref{lambda-orientable:corollary}.

Finally, by applying Proposition \ref{extend:prop} and Lemma \ref{facet:lemma}, we obtain a four-dimensional polytope $\widetilde{\scrP}$ with an orientable $\matF^5_2$-colouring, that induces a non-orientable colouring on one of its facets, which is isometric to $\mathscr{P}$. Indeed, the non-orientable $\matF^3_2$-colouring of each $\scrD_i$ can be extended by Proposition \ref{extend:prop} to an orientable $\matF^5_2$-colouring of a $120$-cell $\scrZ_i$. Then, by Lemma \ref{facet:lemma}, $\scrP = \scrD_1\# \dots \#\scrD_n$ is a facet of a polytope $\widetilde{\scrP} = \scrZ_1\# \dots \#\scrZ_n$. Since the colourings of $\scrD_i$'s match under taking connected sums, the colourings of $\scrZ_i$'s also match. The polytope $\widetilde{\scrP}$ gives rise to a covering manifold with totally geodesic boundary, as described in the proof of Theorem~\ref{theorem}.
\end{example}

\end{document}